\documentclass{amsart}

\usepackage{scalerel}
\usepackage[only,llbracket,rrbracket,llparenthesis,rrparenthesis]{stmaryrd}
\usepackage{tikz}

\numberwithin{equation}{section}

\newtheorem{theorem}{Theorem}[section]

\newtheorem{proposition}[theorem]{Proposition}

\newtheorem{corollary}[theorem]{Corollary}

\newtheorem{lemma}[theorem]{Lemma}

\theoremstyle{definition}

\newtheorem{remark}[theorem]{Remark}
\newtheorem{example}[theorem]{Example}

\def\doubleunderline#1{\underline{\underline{#1}}}

\def\endproof{\hfill$\square$\medskip}

\def\ZZ{\mathbb{Z}}

\def\PP{\mathcal{P}}

\def\ii{\mathbf{i}}
\def\jj{\mathbf{j}}

\def\sgn{\operatorname{sgn}}

\dedicatory{To George Andrews on the occasion of his 80th birthday}

\begin{document}

%\hskip 12cm {\bf not for circulation}
%\medskip

\title%[Cluster algebras X]
{Noncommutative Catalan numbers}

\author{Arkady Berenstein}
\address{\noindent Department of Mathematics, University of Oregon,
Eugene, OR 97403, USA} \email{arkadiy@math.uoregon.edu}

\author{Vladimir Retakh}
\address{Department of Mathematics, Rutgers University, Piscataway, NJ 08854, USA}
\email{vretakh@math.rutgers.edu}

\thanks{This work was partially supported by the NSF grant~DMS-1403527 (A.~B.)}

\begin{abstract} The goal of this paper is to introduce and study {\it noncommutative Catalan numbers} $C_n$ which belong to the free Laurent polynomial algebra ${\mathcal L}_n$ in $n$ generators.  Our noncommutative numbers admit interesting (commutative and noncommutative) specializations, one of them related to  Garsia-Haiman $(q,t)$-versions, another -- to solving noncommutative quadratic equations.  We also establish total positivity of the corresponding (noncommutative) Hankel matrices $H_n$ and introduce accompanying  {\it noncommutative binomial coefficients} ${\stretchleftright{\llparenthesis}{\begin{matrix} n \\ k \end{matrix}}{\rrparenthesis}}\in {\mathcal L}_{n+k-1},{\stretchleftright{\llparenthesis}{\begin{matrix} n \\ k \end{matrix}}{\rrparenthesis}}'\in {\mathcal L}_n$.

\end{abstract}

\maketitle

\tableofcontents

\section{Introduction}

Catalan numbers $c_n=\frac{1}{n+1}\binom{2n}{n}$, $n\ge 0$ are important combinatorial objects which satisfy a number of remarkable properties such as:

$\bullet$ The recursion $c_{n+1}=\sum\limits_{k=0}^n c_k c_{n-k}$ for all $n\ge 0$ (with $c_0=c_1=1$).

$\bullet$ the determinantal identities $\det\begin{pmatrix}
c_m&c_{m+1}&\dots &c_{m+n}\\
c_{m+1}&c_{m+2}&\dots &c_{m+n+1}\\
 & & \dots & \\
c_{m+n}&c_{m+n+1}&\dots &c_{m+2n}
\end{pmatrix}=1$ for $n\ge 0$, $m\in\{0,1\}$. 
%Recall that the {\it content} of a pair $z=(i,j)$, $i,j\in \ZZ$ is the difference
%$i-j$ and denote the difference by $cont(i,j)$. Let $X=(x_0,x_1,x_2,\dots)$ be a sequence of free
%variables. 

Catalan numbers admit various $q$-deformations (\cite{A,FH,H}) and  $(q,t)$-deformations (\cite{GH,GHa,H}).

In this paper we introduce and study {\it noncommutative Catalan numbers} $C_n$, $n\ge 1$ which are totally noncommutative Laurent polynomials in $n$ variables and satisfy analogues of the recursion and the determinantal identities (Proposition \ref{pr:recursion C} and equation \eqref{eq:principal quasi-minor}). It turns out that specializing these variables to appropriate powers of $q$, we recover Garsia-Haiman $(q,1)$-Catalan numbers. 
Catalan numbers also satisfy a combinatorial identity (formula (4.9) in \cite{C}) involving their truncated counterparts $c_n^k=\binom{n+k}{k}-\binom{n+k}{k-1}$ (so that $c_n=c_n^n=c_n^{n-1}$):
\begin{equation}
\label{eq:classical catalan via truncated catalan}
c_n=\sum\limits_{\substack{a,b\in \ZZ_{\ge 0}:\\a+b\leq n, a-b=d}}
c_{n-b}^a c_{n-a}^b
\end{equation}
for each $n\in \ZZ_{\ge 0}$ and each $d\in \ZZ$ with $|d|\le n$ (e.g., the right hand side does not depend on $d$). 
%The truncated Catalan numbers $c_n^k$ are
%equal to numbers $f(n+1,k+1)$ defined in \cite {C} and the
%identity \eqref{eq:classical catalan via truncated catalan}  corresponds to the formula (4.9) in \cite
%{C}. 
A $q$-deformation of $c_n^k$ was discussed in \cite {CR} under the name of $q$-ballot numbers.

We introduce noncommutative analogues of  truncated Catalan numbers and establish a noncommutative version of \eqref{eq:classical catalan via truncated catalan} (Theorem \ref{th:catalan via truncated catalans}). It is curious that the $c_n^k$ satisfy three more combinatorial identities, two of which involve binomial coefficients:
\begin{equation}
\label{eq:three identities}
c_{n+1}^k=\sum\limits_{j=0}^k c_j c_{n-j}^{k-j},~\sum\limits_{j=0}^k (-1)^j c_{n+k-j}^j\cdot \binom{n-j}{k-j}=0, ~c_{m+n}^k=\sum\limits_{\ell=0}^n c_{m+\ell}^{k-\ell}\cdot \binom{n}{\ell}\ ,
\end{equation}
where $0\le k<n$ in the first two identities and $0\le k\le m+n$ in the third one.
% for all $m,n,k\in \ZZ_{\ge 0}$.

We establish a noncommutative generalization of the first identity \eqref{eq:three identities} (Proposition \ref{pr:recursion Cnk}(c)), 
define appropriate noncommutative versions  ${\stretchleftright{\llparenthesis}{\begin{matrix} n \\ k \end{matrix}}{\rrparenthesis}}$ and  ${\stretchleftright{\llparenthesis}{\begin{matrix} n \\ k \end{matrix}}{\rrparenthesis}}'$ 
of  binomial coefficients and establish analogues of the last two identities \eqref{eq:three identities} with these coefficients (Corollary \ref{cor:mult truncated catalan} and Theorem \ref{th:alternating recursion}) as well as an analogue of the multiplication law for both kinds of noncommutative binomial coefficients (Theorem \ref{th:mult binom}). 

In fact, these constructions and results extend our previous work on Noncommutative Laurent Phenomenon (\cite{BR,BR2}) and we expect more such Phenomena to emerge in Combinatorics, Representation Theory, Topology and related fields.

The paper is organized as follows: Section \ref{sect:main} contains notation and main results and the proofs are given in Section \ref{sect:proofs}.

\subsection{Acknowledgments} This work was partly done during our visits to
Max-Planck-Institut f\"ur Mathematik and Institut des Hautes \'Etudes
Scientifiques. We gratefully acknowledge the support of these
institutions. We thank Philippe Di Francesco and Rinat Kedem for their comments on the first version of the paper, particularly for explaining to us a relationship between noncommutative Stieltjes continued fractions and our noncommutative Catalan series (see Remark \ref{rem:DifK}).

\section{Notation and main results}

\label{sect:main}

Let $F$ be the free group generated by $x_k$, $k\in \ZZ_{\ge 0}$ and  $F_m$ be the (free) subgroup of $F$ generated by $x_0,\ldots,x_m$. 

Denote by $\tilde \PP_n$ the set of all monotonic lattice paths in $[0,n]\times [0,n]$ from $(0,0)$ to $(n,n)$. Clearly, $|\tilde \PP_n|=\binom{2n}{n}$. We say that $P\in \tilde \PP_n$ is {\it Catalan} if for each point $p=(p_1,p_2)\in P$ one has $c(p)\ge 0$, where $c(p_1,p_2):=p_1-p_2$ is the content of $p$. Denote by $\PP_n\subset \tilde \PP_n$ the set of all Catalan paths in $[0,n]\times [0,n]$. Clearly, $|\PP_n|=\frac{1}{n+1}\binom{2n}{n}$ is the $n$-th Catalan number, which justifies the terminology.

We say that a point $p=(p_1,p_2)$ of $P\in \tilde \PP_n$ is a {\it southeast (resp. northwest) corner} of $P$  if $(p_1-1,p_2)\in P$ and $(p_1,p_2+1)\in P$ (resp. $(p_1,p_2-1)\in P$ and $(p_1+1,p_2)\in P$).

 To each $P\in \PP_n$ we assign an element $M_P\in F_n$ by
\begin{equation}
M_P=\overset{\longrightarrow}\prod x_{c(p)}^{\sgn(p)}\ ,
\end{equation}
where the product is  over all corners $p\in P$ (taken in the natural order)
and 
$\sgn(p)=\begin{cases}
1 & \text{if $p$ is southeast}\\
-1& \text{if $p$ is northwest}
\end{cases}$.

\begin{figure}[h!]
\centering
\begin{tikzpicture} [yscale=0.91,xscale=0.91]
\draw [help lines] (0,0) grid (3,3);
%\draw [<->] (5,0) -- (0,0) -- (0,5);
\draw[thick] (0,0) -- (2,0) -- (2,2) -- (3,2) -- (3,3);
%\draw[thick] (2,0) -- (2,2);
%\draw[thick] (2,2) -- (3,2);
%\draw[thick] (2,2) -- (3,2);
%\draw[dashed,ultra thick]
%(1.5,3.5) to [out=-80,in=135] (2.5,1.5);
%\draw [dashed,ultra thick]
%(2.5,1.5) to [out=-45,in=160] (4.2,0.5);
\draw [<->] (0,3.2) -- (0,0) -- (3.2,0);

\node at (0,-0.2) {(0,0)};
\node at (2,-0.2) {(2,0)};

\node at (2,2.2) {(2,2)};
\node at (3,1.8) {(3,2)};
\node at (3,3.2) {(3,3)};
\end{tikzpicture}
\caption{$M_P=x_2x_0^{-1}x_1$ for the above path $P\in \PP_3$}
\end{figure}

We define the {\it noncommutative Catalan number} $C_n\in \ZZ F_n$ by
\begin{equation}
\label{eq: main def}
C_n=\sum_{P\in \PP_n} M_P\ .
\end{equation}

Clearly, under the counit homomorphism $\varepsilon:\ZZ F\to \ZZ$ ($x_k\mapsto 1$) the image $\varepsilon(C_n)$ is $|\PP_n|$, the ordinary Catalan number. 

%So, $\varepsilon (C_2)=2$, $\varepsilon (C_3)=5$, $\varepsilon (C_4)=14$, and so on

Noncommutative Catalan numbers exhibit some symmetries, the first of which is  an anti-automorphism $\overline{\cdot}$ of $\ZZ F$ such that 
$\overline x_k=x_k$ for $k\in \ZZ_{\ge 0}$.
 
\begin{proposition}
\label{pr:bar invariance}
$\overline C_n=C_n$ for all $n\ge 0$.
\end{proposition}

%We prove Proposition \ref{pr:bar invariance} in the beginning of Section %\ref{subsect:3.1}.
\begin{proof} 
Define an involution $s_n:\ZZ^2\to \ZZ^2$ by $s_n(x,y)=(n-y,n-x)$. Clearly, $s_n(\PP_n)=\PP_n$. It is easy to see that 
\begin{equation}
\label{eq:barMP}
\overline M_P=M_{s_n(P)}
\end{equation} 
for all $P\in \PP_n$. Therefore, 
$\overline C_n=\sum\limits_{P\in \PP_n} \overline M_P=\sum\limits_{P\in \PP_n} M_{s_n(P)}=\sum\limits_{P\in \PP_n} M_P=C_n$
for all $n\ge 0$.

The proposition is proved. 
\end{proof}

\begin{example}$C_0=x_0$, $C_1=x_1$, $C_2=x_2+x_1x_0^{-1}x_1$, 
 %\begin{equation}
$$C_3=x_3+x_2x_1^{-1}x_2+x_2x_0^{-1}x_1+x_1x_0^{-1}x_2+x_1x_0^{-1}x_1x_0^{-1}x_1\ ,$$
%\end{equation} 
$$
C_4=x_4+x_3x_2^{-1}x_3+x_2x_0^{-1}x_2+x_3x_1^{-1}x_2+x_2x_1^{-1}x_3+x_3x_0^{-1}x_1+x_1x_0^{-1}x_3
+x_2x_1^{-1}x_2x_1^{-1}x_2$$
$$+x_1x_0^{-1}x_2x_0^{-1}x_1+x_2x_1^{-1}x_2x_0^{-1}x_1+x_1x_0^{-1}x_2x_1^{-1}x_2
+x_2x_0^{-1}x_1x_0^{-1}x_1+x_1x_0^{-1}x_1x_0^{-1}x_2+x_1x_0^{-1}x_1x_0^{-1}x_1x_0^{-1}x_1\ .
$$
\end{example}

%\begin{remark} Note that monomials with alternating powers $\pm 1$ appeared in our paper \cite{BR}.
%\end{remark}

%For a matrix $A=(a_{ij})$ over $\ZZ F$ denote by $T(A)$ the matrix $(T(a_{ij}))$.

It turns out that our noncommutative Catalan numbers satisfy the following generalization of the well-known classical recursion, 
which we prove in Section \ref{subsect:3.1}. 

\begin{proposition}
\label{pr:recursion C}
For $n\ge 0$ one has
\begin{equation}
\label{eq:recursion C}
C_{n+1}=\sum _{k=0}^n C_kx_0^{-1}T(C_{n-k}), ~ C_{n+1}=\sum _{k=0}^n T(C_k)x_0^{-1}C_{n-k} 
\end{equation}
for all $n\in \ZZ_{\ge 0}$,
where $T:\ZZ F\to \ZZ F$ is an endomorphism of $\ZZ F$ given by 
$T(x_k)=x_{k+1}$ for all $k\in \ZZ_{\ge 0}$.

\end{proposition}

For example, $C_2=T(C_1)+C_1x_0^{-1}T(C_0)$ and $C_3=T(C_2)+C_1x_0^{-1}T(C_1)+C_2x_0^{-1}T(C_0)$.

The following is an immediate corollary of Proposition \ref{pr:recursion C}.
\begin{corollary} 
\label{cor:C series}
The formal power series ${\bf C}(t)=\sum\limits_{n=0}^\infty C_nt^n\in (\ZZ F)[[t]]$ satisfies:
\begin{equation}
\label{eq:C series-1}
{\bf C}(t)=x_0+t{\bf C}(t)x_0^{-1}T({\bf C}(t)),~T({\bf C}(t))x_0^{-1}{\bf C}(t)={\bf C}(t)x_0^{-1}T({\bf C}(t)) \ ,
\end{equation}

\end{corollary}

\begin{remark} Applying $\varepsilon$ to \eqref{eq:C series-1}, we obtain the well-known functional equation $c(t)=1+tc(t)^2$ for the classical generating function 
$c(t)=\sum\limits_{n=0}^\infty \varepsilon(C_n)t^n$ of Catalan numbers.
%, or equivalently, $\varepsilon (C_{n+1})=\sum\limits_{k=0}^n \varepsilon (C_k)\varepsilon (C_{n-k})$.
\end{remark}

\begin{remark} 
\label{rem:DifK}
After the first version of this paper became available, Philippe Di Francesco and Rinat Kedem pointed to us that  ${\bf C}(t)x_0^{-1}$ is a {\it noncommutative Stieltjes continued fraction} which can be computed by combining methods of \cite[Section 3.3.1]{DiFK} and \cite[Section 8]{GCLLRT} as follows.
$${\bf C}(t)x_0^{-1}=\lim_{k\to \infty} {\bf S}(x_1x_0^{-1},\ldots,x_kx_{k-1}^{-1},t)\ ,
%(1-t(\dots (1-t(1-t\dots )^{-1}x_kx_{k-1}^{-1})^{-1}x_{k-1}x_{k-2}^{-1})^{-1}\cdots x_2x_1^{-1})x_1x_0^{-1})^{-1}
$$
where ${\bf S}(z_1,t)=(1-z_1t)^{-1}$, ${\bf S}(z_1,\ldots,z_k,t)={\bf S}(z_1,\ldots,z_{k-2},{\bf S}(z_k,t)z_{k-1},t)$ for $k\ge 2$.

\end{remark}

\begin{remark} 
\label{rem:recursion T^2}
In fact, there is another recursion
$$
C_{n+1}=C_nx_0^{-1}x_1 + \sum_{k=1}^n C_kx_1^{-1}T^2(C_{n-k})=x_1x_0^{-1}C_n + \sum_{k=0}^{n-1} T^2(C_k)x_1^{-1}C_{n-k}
$$ 
for $n\ge 1$. For instance,
$
C_3=C_2x_0^{-1}x_1+C_1x_1^{-1}T^2(C_1)+C_2x_1^{-1}T^2(C_0)=C_2x_0^{-1}x_1+x_3+C_2x_1^{-1}x_2$.
The recursion leads to the functional equation
%\begin{equation}
%\label{eq:C series-2}
${\bf C}(t)=x_0+t({\bf C}(t)x_0^{-1}x_1-x_0x_1^{-1}T^2({\bf C}(t))+{\bf C}(t)x_1^{-1}T^2({\bf C}(t)))$, 
%\end{equation}
which we leave as an exercise to the reader.

\end{remark}

\begin{remark} 
%Specializing $t$ to the lower triangular $\ZZ_{\ge 0}\times \ZZ_{\ge 0}$ Jacobi matrix with entries $f_{ij}=\delta_{i,j-1}$, 
%we see that the functional equation  \eqref{eq:C series-1} can be rewritten in a matrix format:
Equations  \eqref{eq:recursion C} can be written in a matrix form: $Hx_0^{-1}T(H)=T(H)x_0^{-1}H=H'$, 
where $H$ (resp. $H'$) is the lower triangular $\ZZ_{\ge 0}\times \ZZ_{\ge 0}$ Toeplitz matrix whose $(i,j)$-th entry is $C_{i-j}$ (resp. $C_{i-j+1}$) if $i\ge j$. Thus, $H^{-1}$ is a lower triangular  Toeplitz  matrix whose $(i,j)$-th entry is $-x_0^{-1}T(C_{i-j-1})x_0^{-1}$ for $i>j$.

\end{remark}

%Let $F_2$ be the free group with generators $x_0,x_1$. 
It turns out that there is a remarkable specialization $\underline C_n\in \ZZ F_1$ of $C_n$. Indeed,  
let $\sigma: \ZZ F\to \ZZ F_1$ be a ring homomorphism  given 
by $\sigma (x_k)=x_0^k x_1^k$, $k\in \ZZ_{\ge 0}$. Abbreviate $\underline C_n:=\sigma (C_n)$ for $n\ge 0$.
%let $\alpha: \ZZ F\to \ZZ F_2$ be a ring homomorphism  given 
%by $\alpha (x_k)=x_0^{k-1}x_1^k$, $k\in \ZZ_{\ge 0}$. Then let $ \underline C_n:=x_0\alpha(C_n)$ for $n\ge 0$.
%We extend $\alpha $ as a ring homomorphism$\ZZ F\to \ZZ F_2$. 
%We also identify the ring of noncommutative polynomials $\ZZ <x_0,x_1>$ with the subring of $\ZZ F_2$.

The following result asserts, in particular, that $\underline C_n$ are noncommutative polynomials (rather than Laurent polynomials) and they 
satisfy yet another noncommutative generalization of the well-known classical recursion for Catalan numbers.

\begin{proposition} 
\label{pr:recursion underline C}
The elements $\underline C_n\in \ZZ\langle x_0,x_1\rangle$ are determined by the following recursion: $\underline C_0=1$ and
\begin{equation}
\label{eq:recursion underline C}
\underline C_{n+1}=\sum_{k=0}^n \underline  C_kx_0\underline  C_{n-k}x_1=\sum_{k=0}^n x_0\underline  C_kx_1\underline  C_{n-k} \ ,
%\underline C_{n+1}=\sum_{k=0}^n \underline  C_k x \underline  C_{n-k}y \ ,
\end{equation}
for $n\ge 0$. In particular, all $\underline C_n$ belong to the free semi-ring 
$\ZZ_{\ge 0} \langle x_0,x_1\rangle\subset \ZZ_{\ge 0} F_1$.
\end{proposition}

Our proof of the proposition is based on the identity $\sigma(T^iC_n)=x_0^i\sigma(C_n)x_1^i$ for $i,n\ge 0$ (see Lemma \ref{le:sigmaT}).

\begin{remark} Applying $\sigma$ to the recursions from Remark \ref{rem:recursion T^2} and using the same argument from the proof of Proposition \ref{pr:recursion underline C}, we obtain another recursion for $\underline C_n$:
$$
\underline C_{n+1}=\underline C_nx_0 x_1 + \sum_{k=1}^n \underline C_kx_1^{-1}x_0\underline C_{n-k} x_1^2=x_0x_1C_n + \sum_{k=0}^{n-1} x_0^2\underline  C_kx_1 x_0^{-1}\underline  C_{n-k}\ . 
$$

\end{remark}

\begin{remark} One can show that the ``two-variable''noncommutative Catalan numbers are invariant under the anti-involution of $\ZZ F_1$ interchanging $x_0$ and $x_1$.

\end{remark}

In fact, we can explicitly compute each $\underline C_n$. Indeed, assign  a monomial $\underline M_P\in F_1$ to each $P\in \PP_n$ by:
$$\underline M_P=x_0^{j_0}x_1^{j_1}x_0^{j_2}\cdots x_1^{j_{2k}}\ ,$$
where $(j_0,j_1,\ldots,j_{2k})\in \ZZ_{> 0}^{2k+1}$ is the sequence of jumps of the path $P$, i.e., the $r$-th northwest corner is $(j_0+j_2+\cdots+j_{2r},j_1+j_3+\cdots +j_{2r+1})$ and $r$-th southeast corner of $P$ is 
$(j_0+j_2+\cdots +j_{2r},j_1+j_3+\cdots+j_{2r-1})$
One can easily see that $\sigma(M_P)=\underline M_P$, so we obtain the following immediate corollary.

\begin{corollary} $\underline C_n=\sum\limits_{P\in \PP_n} \underline M_P$ for all $n\ge 1$.

\end{corollary} 

\begin{example} 
%$\underline C_0= x_0$, 
%$\underline C_1= x_0x_1$, 
$\underline C_2= x_0^2x_1^2 + x_0x_1x_0x_1$,
%\newline
$\underline C_3 = x_0^3x_1^3 +  x_0^2x_1x_0x_1^2 + x_0^2x_1^2x_0x_1 + x_0x_1x_0^2x_1^2 + x_0x_1x_0x_1x_0x_1$,
$$
\underline C_4=x_0^4x_1^4+x_0^3x_1x_0x_1^3+x_0^2x_1^2x_0^2x_1^2+x_0^3x_1^2x_0x_1^2+x_0^2x_1x_0^2x_1^3+x_0^3x_1^3x_0x_1+x_0x_1x_0^3x_1^3
+x_0^2x_1x_0x_1x_0x_1^2$$
$$+x_0x_1x_0^2x_1^2x_0x_1+x_0^2x_1x_0x_1^2x_0x_1+x_0x_1x_0^2x_1x_0x_1^2
+x_0^2x_1^2x_0x_1x_0x_1+x_0x_1x_0x_1x_0^2x_1^2+x_0x_1x_0x_1x_0x_1x_0x_1\ .
$$
\end{example}

%\begin{remark} Actually, this result was the starting point of the project, which emerged out of our previous work \cite{BRRZ} on the %inversion of the  the series $\sum\limits_{n=0}^\infty x_0^n x_1^n$ (?? specialized at ?? with $x_0=a$, $x_1=b$).
%\end{remark}

The following immediate result is a ``two-variable" version of Corollary \ref{cor:C series}.

\begin{corollary} The formal power series $\underline{\bf C}(t)=\sum\limits_{n=0}^\infty \underline C_n t^n\in \ZZ \langle x_0,x_1\rangle[[t]]$ satisfies:
\begin{equation}
\label{eq:C series x0 x1}
\underline {\bf C}(t)=1+t\underline {\bf C}(t)x_0\underline {\bf C}(t)x_1\ .
\end{equation}
%\begin{corollary} The formal power series $\underline{\bf C}=\sum\limits_{n=0}^\infty \underline C_n t^n\in (\ZZ \langle %x_0,x_1\rangle)[[t]]$ satisfies:
%\begin{equation}
%\label{eq:C series x0 x1}
%\underline {\bf C}=1+t\underline {\bf C}x_0\underline {\bf C}x_1\ .
%\end{equation}
\end{corollary}

%This, in particular, allows to solve general noncommutative quadratic equations (see also ??).
\begin{remark} For $t=1$, the equation \eqref{eq:C series x0 x1} coincides with the quadratic equation on formal series $K(x_0,x_1)$ studied in \cite{PPR} where a solution of this equation
was presented as a ``noncommutative Rogers-Ramanujan continued fraction".
%$$\underline {\bf C}=(1-tx\underline {\bf C}y)^{-1}$$

\end{remark} 

\begin{remark} In our previous work \cite{BRRZ} on the inversion of $\sum\limits_{n\geq 0} x_0^n x_1^n$ in the ring of formal series $\ZZ\langle \langle x_0,x_1\rangle\rangle$ in noncommutative variables $x_0,x_1$ we encountered a quadratic equation $D=1-Dx_0x_1+Dx_0Dx_1$ for some 
$D\in \ZZ\langle \langle x_0,x_1\rangle\rangle$
and noticed that it is very similar to \eqref{eq:C series x0 x1}. This was the starting point of the project.
\end{remark}

\begin{remark} In fact, there is another group homomorphism $\pi: F\rightarrow F_1$ given by 
$\pi(x_k)=x_0\cdot (x_0^{-1}x_1)^k$, $k\in \ZZ_{\ge 0}$, which results in an ``almost commutative'' specialization of noncommutative Catalan numbers: 
$\pi(C_n)=\pi(x_n)\cdot \frac{1}{n+1}\binom{2n}{n}$.

\end{remark}

%We will get equation \ref{eq:C series x0 x1} now.

%\begin{proof} From Corollary ... one has
%$$
%x_0^{-1}{\bf C}=1+t(x_0^{-1}{\bf C})x_0^{-1}x_1T_1(x_0^{-1}{\bf C})\ .
%$$

%By the previous lemma, $\sigma (T_1(x_0^{-1}{\bf C}))=x_1^{-1}\sigma ({\bf C})b$. Applying $\sigma $ to the last %equation, we get
%$\underline {\bf C}=1+t\underline {\bf C}x_0\underline {\bf C}x_1$. By setting $x=x_0, y=x_1$ we get euqation \ref %{eq:C series x0 x1}.
%By the previous lemma, $\alpha (T_1(x_0^{-1}{\bf C}))=x_1^{-1}\alpha ({\bf C})b$. Applying $\alpha$ to the last equation, we get
%$Y=1+tYx_0Yx_1$.
%\end{proof}

%CONNECTION WITH OUR PAPER ON INVERSION OF DOUBLE SUMS ????

For each $0\le k\le n$  denote by $\mathcal \PP_n^k$,
the set of all $P\in \PP_n$ such that the rightmost southeast corner $p$ of $P$ satisfies $p=(n,y)$, where $y\le k$. In particular, 
$\PP_n^{n-1}=\PP_n^n=\PP_n$. For each $0\le k\le n$ define {\it truncated noncommutative Catalan number}  $C_n^k\in \ZZ F_n$  by
$$C_n^k:=\sum_{P\in \PP_n^k} M_P\ .$$
% the sum of all monomial $m_P$ where $P\in \mathcal P_{n,k}$. Note that $C_{n,1}=C_n$.

The following recursion on $C_n^k$ is immediate.

\begin{lemma} $C_n^k=C_n^{k-1}+C_{n-1}^kx_{n-k-1}^{-1}x_{n-k}$ for all $1\le k\le n$ (with the convention $C_n^{\ell}=0$ if $\ell>n$). 

\end{lemma}

\begin{example} $C_n^0=x_n$, $C_n^{n-1}=C_n^n=C_n$ for all $n\ge 1$. Also, $C_n^1=x_n+\sum\limits_{i=1}^{n-1} x_ix_{i-1}^{-1}x_{n-1}$, 
%$$
%C_n^2=C_n^1+\sum_{i=2}^{n-1}x_ix_{i-2}^{-1}x_{n-2}+\sum_{1\le i\le j\le n-2}x_ix_{i-1}^{-1}x_jx_{j-1}^{-1}x_{n-2}\ .
%$$
$$
C_n^2=\sum_{1\le i\le j\le n,j>1}x_ix_{i-1}^{-1}x_{j-1}x_{j-2}^{-1}x_{n-2}\ .
$$
\end{example}
%
%$y_n=x_nx_{n-1}^{-1}$
%
%$\tilde C_n^k=C_n^kx_{n-k}^{-1}$
%
%$\tilde C_n^k=\tilde C_n^{k-1}y_{n-k+1}+\tilde C_{n-1}^k$
%

Sometimes it is  convenient to express $C_n^k$  via
$y_i:=x_ix_{i-1}^{-1}$, $i\in \ZZ_{\ge 1}$. Indeed,
%to each path $P\in \PP_n^k$ we assign an element $\tilde M_P$ in the (free) submonoid
%of $F_n$ generated by $y_1,\ldots,y_n$ as follows:
%$$
%\tilde M_P=\overset{\longrightarrow}\prod y_{c(p),c(p^+)}\ ,
%$$
%where the product is  over all southeast corners $p\in P$ (taken in the natural order),
%$p^+$ denotes the immediately following northwest corner, and we abbreviated 
%$y_{b,a}:=y_by_{b-1}\cdots y_{a+1}$ for all $a<b$.
%Clearly, each $M_P$, $P\in\PP_n^k$ is a noncommutative monomial in $y_1,\ldots,y_n$ of degree $k$.
denote $\tilde C_n^k:=C_n^kx_{n-k}^{-1}$ for $k,n\in \ZZ_{\ge 0}$, $k\le n$.

The following result generalizes a number of basic properties of truncated Catalan numbers. 

\begin{proposition} 
\label{pr:recursion Cnk} For all $0\le k\le n$ one has:
  
(a) %$\tilde C_n^k=\sum\limits_{P\in \PP_n^k} \tilde M_P$ for all
%$k,n\in \ZZ_{\ge 0}$, $k\le n$.
  $\tilde C_n^k=\sum\limits_{j_1\le \ldots \le j_k\le n:\\ j_1\ge 1,\ldots,j_k\ge k}
  y_{j_1}y_{j_2-1}\dots y_{j_k-k+1}$.

(b) $\tilde C_n^k=\tilde C_{n-1}^k+\tilde C_n^{k-1}y_{n+1-k}$  (with the convention $\tilde C_n^{\ell}=0$ if $\ell>n$). 

(c) $\tilde C_{n+1}^k=\sum\limits_{i=0}^k \tilde C_i^iT(\tilde C_{n-i}^{k-i})$.

\end{proposition}
 
A proof follows from Lemmas \ref{le:bijection paths sequences}, \ref{le:Jdecomposition}.   

\begin{example} $\tilde C_n^0=1$, $\tilde C_n^1=y_1+\cdots+ y_n$, and $\tilde C_n^n=\tilde C_n^{n-1}y_1$ for all $n\ge 1$. 
$$
\tilde C_n^2=\sum_{1\le i\le j\le n,j> 1}y_iy_{j-1},~
\tilde C_n^3=\sum _{1\leq i\leq j \leq k\leq n,\ j>1,\ k>2}
y_iy_{j-1}y_{k-2} \ .
$$

%Generalization:
%$$
%\tilde C_n^k=\sum _{1\leq i_1\leq i_2\leq \dots \leq i_k;\ i_j\geq j,\ j=1,2,\dots k}
%y_{i_1}y_{i_2-1}\dots y_{i_s-s+1}\ .
%$$

\end{example}

However, the following recursion is rather non-trivial (and we could not find its classical analogue in the literature).

\begin{theorem} 
\label{th:catalan via truncated catalans}
$C_n=\sum\limits_{\substack{a,b\in \ZZ_{\ge 0}:\\a+b\leq n, a-b=d}}
C_{n-b}^a x_{n-a-b}^{-1}\overline {C_{n-a}^b}$
for each $n\in \ZZ_{\ge 0}$ and each $d\in \ZZ$ with $|d|\le n$ (e.g., the right hand side does not depend on $d$).
 
\end{theorem}

A proof is given by Lemmas \ref{le:C_n^k}--\ref{le:J(a,d)} in Section \ref{subsect:3.1}.

\begin{remark}  In particular, Theorem \ref{th:catalan via truncated catalans} provides another confirmation $\overline{\cdot}$\,-invariance of noncommutative Catalan numbers (established in Proposition \ref{pr:bar invariance}).
\end{remark}

It turns out that the above ``two-variable specialization'' $\sigma$ is also of interest for truncated noncommutative Catalan numbers. Indeed, in the notation as above, denote 
$\underline C_n^k:=\sigma(C_n^k)$ and $\doubleunderline C_n^k:=\underline C_n^k x_1^{k-n}$.

The following is immediate.

\begin{corollary} In the notation of Proposition \ref{pr:recursion underline C}, one has

(a) $\underline C_n^k=\sum\limits_{P\in \PP_n^k} \underline M_P$ for all  $k,n\in \ZZ_{\ge 0}$, $k\le n$.

(b) $\doubleunderline C_n^k=\doubleunderline C_n^{k-1}x_1+\doubleunderline C_{n-1}^kx_0$ for all $1\le k\le n$ (with the convention $\doubleunderline C_n^{\ell}=0$ if $\ell>n$). In particular, each $\doubleunderline C_n^k$ is a noncommutative polynomial in $x_0,x_1$ of degree $n+k$. 

\end{corollary}

\begin{example}
%$\doubleunderline C_n^k=\underline C_n^kx_1^{k-n}$.
$\doubleunderline C_n^0=x_0^n$,
$\doubleunderline C_n^1=x_0^nx_1+\sum\limits_{i=1}^{n-1}x_0^ix_1x_0^{n-i}$,
$
\doubleunderline C_n^2=\doubleunderline C _n^1x_1
+ \sum\limits_{1\leq i\leq j\leq n-1, j> 1}
x_0^ix_1x_0^{j-i}x_1x_0^{n-j}$.
\end{example}

It turns out that our (truncated) noncommutative Catalan numbers $\tilde C_n^k$ admit another specialization into certain polynomials in $\ZZ_{\ge 0}[q]$ defined by Garsia and Haiman in \cite{GH}. Namely, let $\chi_q:\ZZ F \to \ZZ [q,q^{-1}]$ be a ring homomorphism defined by
$\chi_q(x_k)= q^{\frac{k(k-1)}{2}}$ for $k\ge 0$, i.e., $\chi_q(y_k)=q^{k-1}$ for $k\in \ZZ_{\ge 1}$.

Define polynomials $c_n^k(q,t)\in \ZZ_{\ge 0}[q,t]$, $0\le k\le n$ recursively by $c_n^0(q,t)=1$ and
%$$C_n^k(q,t)=q^{\frac{(n-k+1)(n-k)}{2}}\sum_{r=1}^k {r+n-k\brack r}_q t^{k-r} C_{k-1}^{k-r}(q,t)$$
%$$\tilde C_n^k(q,t)=q^{-\frac{(n-k+1)(n-k)}{2}}C_n^k(q,t)$$
$$c_n^k(q,t)=\sum_{r=1}^k \begin{bmatrix} r+n-k \\ r \end{bmatrix}_q  t^{k-r} q^{\frac{r(r-1)}{2}}c_{k-1}^{k-r}(q,t) \ ,$$
where $\begin{bmatrix} n \\ k \end{bmatrix}_q$ denotes the $q$-binomial coefficient $\frac{[n]_q!}{[k]_q! [n-k]_q!}$, $[n]_q!=[1]_q\cdots [n]_q$, $[k]_q=\frac{1-q^n}{1-q}=1+q+\cdots q^{k-1}$.

These polynomials are closely related to polynomials $H_{n,k}(q,t)$ introduced by Garsia and Haglund (\cite[Equation I.24]{GHa}), namely, 
$c_n^k(q,t)=t^{-k}q^{-\frac{(n+1-k)(n-k)}{2}}H_{n+1,n+1-k}(q,t)$, in particular, $c_n^n(q,t)=c_n(q,t)$ is the celebrated $(q,t)$-Catalan number introduced in \cite{GH}.

The following result shows that our (truncated) noncommutative Catalan numbers are noncommutative deformations of $(q,1)$-Catalan numbers.
\begin{theorem} 
\label{th:det GH} $\chi_q(\tilde C_n^k)=c_n^k(q,1)$ for all $k\le n$, in particular, $\chi_q(C_n)=c_n(q,1)$ for $n\ge 0$.
  
\end{theorem}

We prove Theorem \ref{th:det GH} in Section \ref{subsect:3.1}.

\begin{example} 
$\chi_q(\tilde C_n^1)=[n+1]_q$ and
$\chi_q(\tilde C_n^k)=\chi_q(\tilde C_n^{k-1})q^{n-k} + \chi_q(\tilde C_{n-1}^k)$ for $1\le k\le n$.

\end{example}

\begin{remark} It is curious that for another class of $q$-Catalan numbers, 
$q^{\frac{n(n-1)}{2}}c_n(q,q^{-1})=\frac{1}{[n+1]_q} \begin{bmatrix} 2n\\ n \end{bmatrix}_q$, there is no analogue of Theorem \ref{th:det GH}. Also, it would be interesting to find an appropriate noncommutative deformations of $(q,t)$-Catalan numbers.

\end{remark}

The following result is a generalization of the well-known property of Hankel determinants of $q$-Catalan numbers.
%This and Theorem \ref{th:hankel minors} imply the following.

\begin{theorem} 
\label{th:Garsia-Haiman det}
For $n\ge 1$, $m\in \{0,1\}$ the determinant of the $(n+1)\times (n+1)$ matrix $(c_{i+j+m}(q,1))$, $i,j=0,\ldots,n$,  is 
$q^{\frac{n(n+1)(4n-1+6m)}{6}}$. 

\end{theorem}

We prove Theorem \ref{th:Garsia-Haiman det} in Section \ref{subsect:3.4}. 

%This result seems to be new.

%Now we introduce noncommutative binomial coefficients ${\stretchleftright{\llparenthesis}{\begin{matrix} n \\ k \end{matrix}}{\rrparenthesis}}\in \ZZ F_{n+k-1}$ and${\stretchleftright{\llparenthesis}{\begin{matrix} n \\ k \end{matrix}}{\rrparenthesis}}'\in ??$.
% in order to compute the inverse matrices of $H_m$, $m\in \{0,1\}$.

%For any subset $J=\{j_1<j_2<\cdots<j_k\}$ of $\ZZ_{\ge 0}$ define $X_J\in F$  by
%$$
%X_I=x_{i_1}^{-1}x_{i_2}x_{i_3}^{-1}x_{i_4}\dots x_{i_k}^{(-1)^k}\ .
%$$
%$$
%X_J=x_{j_1-1}^{-1}x_{j_1}x_{j_2}^{-1}x_{j_2+1}\cdots x_{j_k+k-2}^{-1}x_{j_k+k-1}\ .
%$$

%%Denote the $(ij)$-th entry of matrix $\iota U_m^{\delta}$ by $u_{ij}$. 
%For any set $X$ and  $k\ge 0$ denote by $\left \lbrace \begin{matrix} X \\ k \end{matrix}\right\rbrace$ the
%set of all subsets $J\subset X$ of cardinality $|J|=k$ and define an injective map
%$f:\left \lbrace \begin{matrix} [1,n] \\ k \end{matrix}
%\right\rbrace\to \left \lbrace \begin{matrix} [0,2n-k-1] \\ k \end{matrix}\right\rbrace$
 %by $f(\{j_1 < j_2 < \dots < j_k\})=\{2j_1-2,2j_2-3,\ldots,2j_k-k-1\}$.
%
%Finally, d
Define the {\it noncommutative binomial coefficients} 
${\stretchleftright{\llparenthesis}{\begin{matrix} n \\ k \end{matrix}}{\rrparenthesis}}\in \ZZ F_{n+k-1},{\stretchleftright{\llparenthesis}{\begin{matrix} n \\ k \end{matrix}}{\rrparenthesis}}' \in \ZZ F_n$ by 
%$\displaystyle{\stretchleftright{\llparenthesis}\begin{matrix} n \\ k \end{matrix}\stretchleftright{\rrparenthesis}$
%${\stretchleftright{\llparenthesis}{\begin{matrix} n \\ k \end{matrix}}{\rrparenthesis}}\in \ZZ F_{2n-k-1}$ by 
$${\stretchleftright{\llparenthesis}{\begin{matrix} n \\ k \end{matrix}}{\rrparenthesis}}=
\sum 
%_{J\in \tiny{\left \lbrace \begin{matrix} [1,n] \\ k \end{matrix}\right\rbrace}}  
y_J,~{\stretchleftright{\llparenthesis}{\begin{matrix} n \\ k \end{matrix}}{\rrparenthesis}}'= \sum 
%_{J\in \tiny{\left \lbrace \begin{matrix} [1,n] \\ k \end{matrix}\right\rbrace}}  
y'_J
%,~\ell_{ij}^\delta=x_{2n-k}\overline {S_k^n}=\sum_{J\in S_k^n} \overline  {X_{[0,2n-k-1]\setminus f(J)}}
$$
%Let $J$ be a sequence of numbers $j_1< j_2 < \dots < j_{2i-2+\delta}$. Set
%$f(j_s)=2j_s-(s+1)$ for $1 \leq s \leq i+j-2+\delta$ and set 
%$$
%f(J)=(f(j_s))_{s=1}^{s=i+j-2+\delta}\ .
%$$
where each summation is over all subsets $J=\{j_1<j_2<\cdots<j_k\}$ of $[1,n]$ and we abbreviated $y_J=y_{j_k+k-1}\cdots  y_{j_2+1}y_{j_1}$, $y'_J=y_{j_1+k-1}y_{j_2+k-3}\cdots y_{j_k+1-k}$ 
%and $\overline y_j=x_{j-1}^{-1}x_j$ 
for $j\in \ZZ_{\ge 1}$. 

\begin{remark} 
The $q$-binomial coefficients can be expressed as $\begin{bmatrix}n \\ k\end{bmatrix}_q=\sum q^{j_1+\cdots +j_k-\frac{k(k+1)}{2}}$, where the summation is over all subsets $J=\{j_1<j_2<\cdots<j_k\}$ of $[1,n]$.
Therefore, under the above specialization $\chi_q:\ZZ F\to \ZZ[q,q^{-1}]$ we have $\chi_q\left({\stretchleftright{\llparenthesis}{\begin{matrix} n \\ k \end{matrix}}{\rrparenthesis}}\right)
=q^{k(k-1)}\begin{bmatrix}n \\ k\end{bmatrix}_q$, $\chi_q\left({\stretchleftright{\llparenthesis}{\begin{matrix} n \\ k \end{matrix}}{\rrparenthesis}}'\right)
=q^{\frac{k(k-1)}{2}}\begin{bmatrix}n \\ k\end{bmatrix}_q$ for all $k,n\in \ZZ_{\ge 0}$.

\end{remark}

\begin{example} 
${\stretchleftright{\llparenthesis}{\begin{matrix} n \\ 0 \end{matrix}}{\rrparenthesis}}={\stretchleftright{\llparenthesis}{\begin{matrix} n \\ 0 \end{matrix}}{\rrparenthesis}}'=1$, 
${\stretchleftright{\llparenthesis}{\begin{matrix} n \\ 1 \end{matrix}}{\rrparenthesis}}={\stretchleftright{\llparenthesis}{\begin{matrix} n \\ 1 \end{matrix}}{\rrparenthesis}}'=\sum\limits_{i=1}^n y_i$,
${\stretchleftright{\llparenthesis}{\begin{matrix} n \\ 2 \end{matrix}}{\rrparenthesis}}=\sum\limits_{1\leq i<j\leq n}y_{j+1} y_i$,
${\stretchleftright{\llparenthesis}{\begin{matrix} n \\ n \end{matrix}}{\rrparenthesis}}=y_{2n-1} \cdots y_3y_1=y_{[1,n]}$,
${\stretchleftright{\llparenthesis}{\begin{matrix} n \\ n-1 \end{matrix}}{\rrparenthesis}}=\sum\limits_{i=1}^n  y_{[1,n]\setminus \{i\}}$,
${\stretchleftright{\llparenthesis}{\begin{matrix} n \\ n-2 \end{matrix}}{\rrparenthesis}}=\sum\limits_{1\le i<j\le n} y_{[1,n]\setminus \{i,j\}}$,
${\stretchleftright{\llparenthesis}{\begin{matrix} n \\ 2 \end{matrix}}{\rrparenthesis}}'=\sum\limits_{1\leq i<j\leq n}y_{i+1} y_{j-1}$,
${\stretchleftright{\llparenthesis}{\begin{matrix} n \\ n \end{matrix}}{\rrparenthesis}}'=y_ny_{n-1} \cdots y_1=y'_{[1,n]}$, \ \ 
${\stretchleftright{\llparenthesis}{\begin{matrix} n \\ n-1 \end{matrix}}{\rrparenthesis}}'=\sum\limits_{i=1}^n  y'_{[1,n]\setminus \{i\}}$,
${\stretchleftright{\llparenthesis}{\begin{matrix} n \\ n-2 \end{matrix}}{\rrparenthesis}}'=\sum\limits_{1\le i<j\le n} y'_{[1,n]\setminus \{i,j\}}$.

%, and $\ell_{1j}^1=\sum\limits_{k=0}^{j-1}\overline {X_{[0,2j-2]\setminus \{2k\}}}$.
\end{example}

Clearly, $\varepsilon \left({\stretchleftright{\llparenthesis}{\begin{matrix} n \\ k \end{matrix}}{\rrparenthesis}}\right)=\varepsilon \left({\stretchleftright{\llparenthesis}{\begin{matrix} n \\ k \end{matrix}}{\rrparenthesis}}'\right)=\binom{n}{k}$ and  ${\stretchleftright{\llparenthesis}{\begin{matrix} n \\ k\end{matrix}}{\rrparenthesis}}={\stretchleftright{\llparenthesis}{\begin{matrix} n \\ k\end{matrix}}{\rrparenthesis}}'=0$ if $k\notin [0,n]$.

%The following recursion on ${\stretchleftright{\llparenthesis}{\begin{matrix} n \\ k \end{matrix}}{\rrparenthesis}}$ and ${\stretchleftright{\llparenthesis}{\begin{matrix} n \\ k \end{matrix}}{\rrparenthesis}}'$ is a generalization of the classical property of 
%
%\begin{lemma} 
%\label{le:recursion n choose k}
%
%\end{lemma}

Similarly to the classical case, we have an analogue of the Pascal triangle and the multiplication law for noncommutative binomial coefficients.

% and truncated Catalan numbers.

\begin{theorem} 
\label{th:mult binom}
${\stretchleftright{\llparenthesis}{\begin{matrix} m+n \\ k \end{matrix}}{\rrparenthesis}}=\sum\limits_{\substack{a,b\in \ZZ_{\ge 0}:\\a+b=k}} T^{n+b}\left({\stretchleftright{\llparenthesis}{\begin{matrix} m \\ a \end{matrix}}{\rrparenthesis}}\right) {\stretchleftright{\llparenthesis}{\begin{matrix} n \\ b \end{matrix}}{\rrparenthesis}}$, ${\stretchleftright{\llparenthesis}{\begin{matrix} m+n \\ k \end{matrix}}{\rrparenthesis}}'=\sum\limits_{\substack{a,b\in \ZZ_{\ge 0}:\\a+b=k} }  T^b\left({\stretchleftright{\llparenthesis}{\begin{matrix} m \\ a \end{matrix}}{\rrparenthesis}}'\right) T^{m-a}\left({\stretchleftright{\llparenthesis}{\begin{matrix} n \\ b \end{matrix}}{\rrparenthesis}}'\right)$ for $m,n,k$ $\in \ZZ_{\ge 0}$. In particular,
${\stretchleftright{\llparenthesis}{\begin{matrix} n+1 \\ k \end{matrix}}{\rrparenthesis}}
={\stretchleftright{\llparenthesis}{\begin{matrix} n \\ k \end{matrix}}{\rrparenthesis}}+y_{n+k}{\stretchleftright{\llparenthesis}{\begin{matrix} n \\ k-1 \end{matrix}}{\rrparenthesis}}$, ${\stretchleftright{\llparenthesis}{\begin{matrix} n+1 \\ k \end{matrix}}{\rrparenthesis}}'
=T\left({\stretchleftright{\llparenthesis}{\begin{matrix} n \\ k \end{matrix}}{\rrparenthesis}}'\right)+y_k\left({\stretchleftright{\llparenthesis}{\begin{matrix} n \\ k-1 \end{matrix}}{\rrparenthesis}}'\right) $ for all $n,k\in \ZZ_{\ge 0}$. 

\end{theorem}

Actually, Theorem \ref{th:mult binom} which is proved in Section \ref{subsect:3.2} together with the recursion from Proposition \ref{pr:recursion Cnk}(b) imply the following analogue of the multiplication law for the truncated noncommutative Catalan numbers, which justified the introduction of noncommutative binomial coefficients of the ``second kind.''

\begin{corollary}
\label{cor:mult truncated catalan}
$\tilde C_{m+n}^k=\sum\limits_{\ell=0}^n \tilde C_{m+\ell}^{k-\ell}\cdot T^{m-k+\ell}\left({\stretchleftright{\llparenthesis}{\begin{matrix} n \\ \ell \end{matrix}}{\rrparenthesis}}'\right)$
for all $m,n,k\in \ZZ_{\ge 0}$.

\end{corollary}

The following relation between truncated noncommutative Catalan numbers and the binomial coefficients of the ``first kind'' is rather surprising.
%result is based on a striking similarity of recursions for noncommutative truncated Catalan numbers and their binomial counterparts in Lemma \ref{le:recursion Cnk}(b)
% and \ref{le:recursion n choose k}.

\begin{theorem}
\label{th:alternating recursion} 
%$\sum\limits_{j=k}^n (-1)^j \tilde C_{n+j}^{n-j}\cdot {\stretchleftright{\llparenthesis}{\begin{matrix} j+k \\ j-k \end{matrix}}{\rrparenthesis}}=0$ for any $0\le k< n$.
%$$
%\sum_{k=0}^{i-j} (-1)^kC_{i+j+k}^{i-j-k}\cdot x_{2(j+k)}^{-1}\cdot \overline {2j+k \choose k}=0
%$$
%or
$\sum\limits_{j=0}^k (-1)^j \tilde C_{n+k-j}^j\cdot {\stretchleftright{\llparenthesis}{\begin{matrix} n-j \\ k-j \end{matrix}}{\rrparenthesis}}=0$ 
for any $0< k\le  n$.
\end{theorem}

We prove Theorem \ref{th:alternating recursion} in Section \ref{subsect:3.2} (Lemmas \ref{le:counting summands y}, \ref{le:alternation cancelation}).

\begin{remark} In fact, there is an accompanying identity 
%$\sum\limits_{k'=i'}^{j'} (-1)^{k'+j'} {\stretchleftright{\llparenthesis}{\begin{matrix} i'+k'+m \\ k'-i' \end{matrix}}{\rrparenthesis}}\cdot \tilde C_{m+k'+j'}^{j'-k'} =0$
%$\sum\limits_{j=0}^{k} (-1)^{j} {\stretchleftright{\llparenthesis}{\begin{matrix} i'+j'-j+m \\ k-j \end{matrix}}{\rrparenthesis}}\cdot \tilde C_{m+2j'-j}^{j} =0$
$\sum\limits_{j=0}^k (-1)^j {\stretchleftright{\llparenthesis}{\begin{matrix} n+k-j \\ j \end{matrix}}{\rrparenthesis}}\cdot \tilde C_{n-j}^{k-j} =0$ for any $0< k\le  n$, 
which follows from Theorem \ref{th:inverse gauss} below. We leave this as an exercise to the readers. 
\end{remark}

This turns out to be equivalent to the following ``determinantal'' identities between noncommutative truncated Catalan numbers and binomial coefficients (whose classical analogues also seem to be new).

\begin{theorem} 
\label{th:Cnk via binomials} For all  $k,n\in \ZZ_{\ge 0}$, $k\le n$ one has
$\tilde C_n^k= \sum_J (-1)^{k+1-|J|} {M_{n,J}},
~{\stretchleftright{\llparenthesis}{\begin{matrix} n \\ k \end{matrix}}{\rrparenthesis}}
=\sum_J(-1)^{k+1-|J|} {\tilde M_{n,J}}$, 
where each summation is over all subsets $J=\{0=j_0<\cdots <j_\ell=k\}$ of $[0,k]$ and 
$$M_{n,J}={\stretchleftright{\llparenthesis}{\begin{matrix} n+j_{\ell-1}+j_\ell-k \\ j_\ell-j_{\ell-1} \end{matrix}}{\rrparenthesis}}
\cdots {\stretchleftright{\llparenthesis}{\begin{matrix} n+j_1+j_2-k \\ j_2-j_1 \end{matrix}}{\rrparenthesis}}{\stretchleftright{\llparenthesis}{\begin{matrix} n+j_0+j_1-k \\ j_1-j_0 \end{matrix}}{\rrparenthesis}}
\ ,$$
$$\tilde M_{n,J}=\tilde C_{n+j_0+j_1-k}^{j_1-j_0}\cdot
\tilde C_{n+j_1+j_2-k}^{j_2-j_1} \cdots \tilde C_{n+j_{\ell-1}+j_{\ell}-k}^{j_\ell-j_{\ell -1}}\ .$$

\end{theorem}

We prove Theorem \ref{th:Cnk via binomials} in Section \ref{subsect:3.4}.

Actually, Theorems \ref{th:det GH}, \ref{th:alternating recursion}, and \ref{th:Cnk via binomials} hint to some remarkable properties of Hankel matrices with noncommutative Catalan numbers as entries. 
%To explain this, let us introduce appropriate notation.	
%\begin{proposition} 
%\label{pr:recursion C}
%For $n\ge 1$ one has

%where  $T:\ZZ F\to \ZZ F$ is an endomorphism of $\ZZ F$ given by 
%$$T(x_k)=x_{k+2}$ for all $k\in \ZZ_{\ge 0}$. 
%+C_{n/2}x_1^{-1}C_{n/2}'\ .

%\end{proposition} 

For $m\in \ZZ_{\ge 0}$ define the $\ZZ_{\ge 0}\times \ZZ_{\ge 0}$ matrix  $H_m$ over $\ZZ F$ 
whose $(i,j)$-th entry is $C_{m+i+j}$, $i,j\in \ZZ_{\ge 0}$
% (with the convention $C_k=0$ whenever $k<0$) 
and for each $n\ge 0$ denote by  $H_{m,n}$ the principal $[0,n]\times [0,n]$ submatrix of $H_m$. 
%whose $(i,j)$-th entry is $C_{m+i+j}$, $i,j=0,\ldots,n$ (with the convention $C_k=0$ whenever $k<0$). 

\begin{example}
$
%H_{-1}^1=\left (\begin{matrix} 0&C_0\\C_0&C_1\end{matrix} \right ), \
H_{0,1}=\left ( \begin{matrix} C_0&C_1\\C_1&C_2\end{matrix} \right ), \
H_{1,1}=\left (\begin{matrix} C_1&C_2\\C_2&C_3\end{matrix} \right )$,
$
%H_{-2}^2=\left (\begin{matrix} 0&0&C_0\\0&C_0&C_1\\C_0&C_1&C_2\end{matrix} \right ), \
%H_{-1}^2=\left (\begin{matrix} 0&C_0&C_1\\C_0&C_1&C_2\\C_1&C_2&C_3\end{matrix} \right ), \
H_{0,2}=\left ( \begin{matrix} C_0&C_1&C_2\\C_1&C_2&C_3\\C_2&C_3&C_4\end{matrix} \right ),\ 
H_{1,2}=\left ( \begin{matrix} C_1&C_2&C_3\\C_2&C_3&C_4\\C_3&C_4&C_5\end{matrix} \right )$.
\end{example}
We refer to all $H_m$ and $H_m^n$ as {\it noncommutative Hankel-Catalan matrices} by analogy with its classical 
counterpart $\varepsilon(H_{m,n})\in Mat_{n+1,n+1}(\ZZ)$. 

We will finish the section by showing that each  $H_{m,n}$, $m\in \{0,1\}$, $n\ge 0$ admits a Gauss factorization over $\ZZ F$ involving truncated noncommutative Catalan numbers and it inverse (which is also a matrix over $\ZZ F$) is given by an interesting combinatorial formula involving our noncommutative binomial coefficients.

%For small indices ??? DO WE NEED IT???
%$$
%C_2=C_1x_0^{-1}x_1+T^2(C_0)\ ,
%$$
%$$
%C_3=C_2x_0^{-1}x_1+ C_2x_1^{-1}T^2(C_0)+T^2(C_1)\ ,
%$$
%$$
%C_4=C_3x_0^{-1}x_1+C_3x_1^{-1}T^2(C_0) + C_2x_1^{-1}T^2(C_1) + T^2(C_2)\ ,
%$$
%$$
%C_5=C_4x_0^{-1}x_1 + C_4x_1^{-1}T^2(C_0) +C_3x_1^{-1}T^2(C_1)+C_2x_1^{-1}T^2(C_2) +T^2(C_3)\ .
%$$

%Set $C_i'=x_{i-2}+C_{i-1}x_i^{-1}C_{i-1}$, $i=2,3$. Then
%$$
%C_i'x_{i-2}^{-1}C_{i-1}=C_{i-1}x_i^{-1}C_i, \ \ 
%C_ix_i^{-1}C_{i-1}=C_{i-1}x_{i-2}^{-1}C_i'\ .
%$$  
  
%Note that under $C_2'$ can be obtained from $C_2$ by exchanging $x_0$ and $x_2$ and $C_3'$ can
%be obtained from $C_3$ under homomorphism exchanging $x_1$ and $x_3$ and sending $x_0$ to $x_3x_1^{-1}x_0x_1^{-1}x_3$. 

%For each $S=\{s_0<s_1<s_2<\cdots s_k\}$ 

%The values of $C_n$ for $x_k=1$, $k\geq 0$ coincide with the corresponding Catalan numbers. Also, in the commutative
%case the quasideterminant of the Hankel matrix $H_n$ of index $(n,n)$ and the first row $(a_0,a_1,\dots,a_n)$ is equal to
%the determinant of $H_n$ divided by the determinant of $H_{n-1}$ and Theorem \ref{th:hankel minors} implies the well-known identities for
%Hankel determinants for Catalan numbers.

%The proof of this theorem is based on the following result which is of interest by itself. 

For $m\in \{0,1\}$ let $L_m$ be the lower unitriangular  $\ZZ_{\ge 0}\times  \ZZ_{\ge 0}$ matrix whose 
$(j,i)$-th entry, $0\le i\le j$, is $\tilde C_{i+j+m}^{j-i}$ and let $U_m$ be the upper triangular  $\ZZ_{\ge 0}\times  \ZZ_{\ge 0}$ matrix whose 
$(i,j)$-th entry, $0\le i\le j$, is $\overline {C_{i+j+m}^{j-i}}$). 

%CHANGE TO SCALAR IDENTITIES, POSTPONE MATRICES UNTIL SECTION 2.

\begin{theorem}
\label{th:Gauss} $H_m=L_m\cdot U_m$ for each  $m\in \{0,1\}$. 
%In particular, each matrix coefficient of $(H_{[m,m+\delta]})^{-1}$ belongs to $\ZZ F_{2m+\delta}$.
%
%For any $m\geq 0$
%$$
%\begin{pmatrix}
%C_0&C_1&\dots &C_n\\
%C_1&C_2&\dots &C_{n+1}\\
%C_2&C_3&\dots &C_{n+2}\\
 %& & \dots & \\
%C_n&C_{n+1}&\dots &C_{2n}
%\end{pmatrix}    
%=
%\begin{pmatrix}
%1&0&0&\dots&0\\
%u_{21}&1&0&\dots&0\\
%u_{31}&u_{32}&1&\dots&0\\
 %& & &\dots & \\
%u_{n1}&u_{n2}&u_{n3}&\dots &1
%\end{pmatrix}
%\begin{pmatrix}
%d_1&0&0&\dots&0\\
%0&d_2&0&\dots&0\\
%0&0&d_3&\dots&0\\
 %& & &\dots & \\
%0&0&0&\dots &d_n
%\end{pmatrix}
%\begin{pmatrix}
%1&v_{12}&v_{13}&\dots&v_{1n}\\
%0&1&v_{23}&\dots&v_{2n}\\
%0&0&1&\dots&v_{3n}\\
 %& & &\dots & \\
%0&0&0&\dots &1
%\end{pmatrix}
%$$
%where $d_i=x_{2(i-1)}$, $u_{im}x_{2(m-1)}=C_{i,m}$, $x_{2(m-1)}v_{mi}=\overline C_{i,m}$
%for $1\leq m<i\leq n$.
\end{theorem}

We prove Theorem \ref{th:Gauss} in Section \ref{subsect:3.3}.

\begin{remark} 
\label{rem:Gauss} 
A classical version of this result, 
$\varepsilon(H_m)=\varepsilon(L_m)\cdot \varepsilon(U_m)$, was established  in \cite{Ai}.

\end{remark}

Theorem \ref{th:Gauss} and \cite[Theorem 4.9.7]{GGRW} imply the following immediate corollary.

\begin{corollary} 
\label{cor:quasidet}
$C_{m+i+j}^{j-i}$ equals the quasidetermiant 
$\left |\begin{matrix}
C_m&C_{m+1}&\dots&C_{m+i}\\
C_{m+1}&C_{m+2}&\dots&C_{m+i+1}\\
 & &\dots& & \\
C_{m+i-1}&C_{m+i}&\dots&C_{m+2i-1}\\
C_{m+j}&C_{m+j+1}&\dots&\boxed {C_{m+i+j}}
\end{matrix} \right |$ for $0\le i\le j$, $m\in \{0,1\}$ (see \cite{gr,gr2} for notation). In particular, 
\begin{equation}
\label{eq:principal quasi-minor}
\left |\begin{matrix}
C_m&C_{m+1}&\dots &C_{m+n}\\
C_{m+1}&C_{m+2}&\dots &C_{m+n+1}\\
 & & \dots & \\
C_{m+n}&C_{m+n+1}&\dots &\boxed {C_{m+2n}}\end{matrix}\right |=x_{m+2n}
\end{equation} for all $n\in \ZZ_{\ge 0}$,  $m\in \{0,1\}$.

\end{corollary}

%%In particular, if $S$ is an interval $[i,j]$ in 
%%$\ZZ_{\ge 0}$, then $H_S$ is a Hankel matrix. 
%For $m,n\ge 0$ denote by $D_m^n$ the {\it quasideterminant} of $H_m^n$ 
%with the marked $(n,n)$-entry
%% (See {\bf Appendix} for notation), i.e., 
%$D_m^n=\left |\begin{matrix}
%C_m&C_{m+1}&\dots &C_{m+n}\\
%C_{m+1}&C_{m+2}&\dots &C_{m+n+1}\\
 %& & \dots & \\
%C_{m+n}&C_{m+n+1}&\dots &\boxed {C_{m+2n}}
%\end{matrix}\right |$ in notation of ??. 

%\begin{theorem} 
%\label{th:hankel determinants} 
%$D_m^n=x_{m+2n}$ 
%%D_{1,n}=\left |\begin{matrix}
%%C_1&C_2&\dots &C_{n+1}\\
%%C_2&C_3&\dots &C_{n+2}\\
%% & & \dots & \\
%%C_{n+1}&C_{n+2}&\dots &\boxed {C_{2n+1}}
%%\end{matrix}\right |
%%=x_{2m+1}\ .
%%$$
%\end{theorem}

\begin{remark} In fact, \eqref{eq:principal quasi-minor} is noncommutative generalization 
of the well-known fact that $\det(\varepsilon(H_{0,n}))=\det(\varepsilon(H_{1,n}))=1$ for $n\ge 0$. Moreover similarly to the classical case, noncommutative Catalan numbers are uniquely determined by equations \eqref{eq:principal quasi-minor} for $n\in \ZZ_{\ge 0}$,  $m\in \{0,1\}$.
%It was our initial approach to noncommutative Catalan numbers.
\end{remark}

\begin{remark} Noncommutative Hankel quasideterminants were introduced in \cite{GCLLRT} in the context of inversion of noncommutative power series. In fact, \cite[Corollary 8.3]{GCLLRT} asserts that such an inverse can be expressed via continued fractions involving such quasideterminants of the coefficients of the series in question. This correlates with Remark \ref{rem:DifK} above.

\end{remark}

%The following is a generalization of Theorem \ref{th:hankel minors}.

%\begin{lemma} %If $S\cap \{0,1\}\ne \emptyset$ then $D_S\in \ZZ_{\ge 0} F$ for each finite subset $S\subset \ZZ_{\ge 0}$. More precisely, ...
%$D_{[\delta,k-1+\delta]\sqcup\{\ell+\delta\}}=C_{k+\ell+\delta}^{\ell-k}$
%for all $0\le k\le \ell$, $\delta\in \{0,1\}$.
%
%\end{lemma}

%For any $n\times n$-matrix $A$ set $\iota (A)=PA^{-1}P$ where $P$ is the diagonal matrix $diag (1,-1,1,\dots, (-1)^n)$.
%We need extra notations to describe the elements of $\iota U_k^{\delta}$.

For $m\in \{0,1\}$ let $L^-_m$ be the lower unitriangular  $\ZZ_{\ge 0}\times  \ZZ_{\ge 0}$ matrix whose 
$(j,i)$-th entry, $0\le i\le j$, is $(-1)^{i+j}{\stretchleftright{\llparenthesis}{\begin{matrix} i+j+m \\ j-i \end{matrix}}{\rrparenthesis}}$ and let $U^-_m$ be the upper triangular  $\ZZ_{\ge 0}\times  \ZZ_{\ge 0}$ matrix whose 
$(i,j)$-th entry, $0\le i\le j$, is $\displaystyle{(-1)^{i+j}\overline{{\stretchleftright{\llparenthesis}{\begin{matrix} i+j+m \\ j-i \end{matrix}}{\rrparenthesis}}}x_{2j+m}^{-1}}$.

%Clearly, $\displaystyle{\varepsilon (U_{k})=\varepsilon (L_{k})^T=\left((-1)^{i+j}\binom{i+j-2+k}{2i-2+k}, 1\le i \le j\right)}$.

For any $\ZZ_{\ge 0}\times \ZZ_{\ge 0}$ matrix $M$ denote by $M|_n$ the principal $(n+1)\times (n+1)$-submatrix of $M$ (e.g., $H_{m,n}=H_m|_n$).
%We abbreviate $L^-_m:=(L_m)^{-1}$ for $Also denote  by $L_m^n$ (resp. $U_m^n$) the principal $n\times n$ submatrix of $L_m$ (resp. of 

\begin{theorem} 
\label{th:inverse gauss}
$(U_m)^{-1}=U^-_m$ and $(L_m)^{-1}=L^-_m$,  hence $(H_{m,n})^{-1}=U^-_m|_n\cdot L^-_m|_n$ 
for $m\in \{0,1\}$, $n\ge 1$.
\end{theorem}

%Denote by $\ell _{ji}^{(m)}$ the elements in the $j$-row and $i$-th column of matrix 
%$\iota (L_[m, m+k])$ ({\bf Define the matrix}). 

%\begin{corollary} 
%For $1\leq j\leq k+1$
%$$
%\ell _{j1}^{(1)}=\sum i=0^{j-1}\overline X_{[0,2j-2]\setminus \{2i\}}\ .
%$$
%\end{corollary} 

\begin{remark} Similar to Remark \ref{rem:Gauss} the classical version of this result, 
$\varepsilon(H_{m,n})^{-1}=\varepsilon(L^-_m|_n)\cdot \varepsilon(U^-_m)|_n$, seems to be new.

\end{remark}

Computation of $H_m^{-1}$ for $m\ge 2$ is a more challenging task, which we will perform elsewhere.

%in Section \ref{sec:hankel inversion}. 

\section{Proofs of main results}

\label{sect:proofs}

\subsection{Proof of Propositions 
%\ref{pr:bar invariance}, 
\ref{pr:recursion C}, \ref{pr:recursion underline C}, \ref{pr:recursion Cnk} and Theorems \ref{th:catalan via truncated catalans}, \ref{th:det GH}}

\label{subsect:3.1} 

%Prove Proposition \ref{pr:bar invariance} first. Define an involution $s_n:\ZZ^2\to %\ZZ^2$ by $s_n(x,y)=(n-y,n-x)$. Clearly, $s_n(\PP_n)=\PP_n$. It is easy to see that 
%\begin{equation}
%\label{eq:barMP}
%\overline M_P=M_{s_n(P)}
%\end{equation} 
%for all $P\in \PP_n$. Therefore, 
%$\overline C_n=\sum\limits_{P\in \PP_n} \overline M_P=\sum\limits_{P\in \PP_n} M_{s_n(P)}=\sum\limits_{P\in \PP_n} M_P=C_n$
%for all $n\ge 0$.

%The proposition is proved. 
%\endproof

%\medskip

%Prove Proposition \ref{pr:bar invariance} now. We proceed by induction in $n$. Indeed, $\overline C_k=C_k$ for $k=0,1$. Since $T$ commutes with $\overline{\cdot}$, we obtain $\overline C_{n+1}=C_{n+1}$ from \eqref{eq:recursion C} and the inductive hypothesis $\overline C_k=C)k$ for $k\le n$.
%
%\endproof
%

We start with a proof of Proposition \ref{pr:recursion Cnk}. Then specializations will lead to Propositions
\ref{pr:recursion C} and \ref{pr:recursion underline C}.

\medskip\noindent
\textbf {Proof of Proposition \ref{pr:recursion Cnk}}.
Prove (a) first. Denote by ${\bf J}_n^k$ the set of all sequences $\jj=(j_1,\ldots,j_k)\in \ZZ^k$ such that $j_1\le \ldots \le j_k\le n$ and $j_1\ge 1,\ldots,j_k\ge k$. 

For each $P\in \PP_n^k$ and $s\in [1,k]$ denote by $j_s(P)$ the minimum of $x$-coordinates of all points in $P$ whose $y$-coordinate is $s$. 
For each $\jj=(j_1,\ldots,j_k)\in \ZZ^k$ with $j_s\ge s$, $s\in [1,k]$ we abbreviate $y_\jj=y_{j_1}y_{j_2-1}\dots y_{j_k-k+1}$.

The following is immediate.

\begin{lemma} 
\label{le:bijection paths sequences}
For all $k,n\in \ZZ_{\ge 0}$, $k\le n$ one has:  

(a) The assignments $P\mapsto \jj(P):=(j_1(P),\ldots,j_k(P))$ defines a bijection $\PP_n^k\widetilde \to {\bf J}_n^k$.

(b)  For each $P\in \PP_n^k$ we have  $M_Px_{n-k}^{-1}=y_{\jj(P)}$.
\end{lemma}

Using Lemma \ref{le:bijection paths sequences}(b), we obtain $\tilde C_n^k=\sum\limits_{\jj\in {\bf J}_n^k} y_\jj$
and thus finish the proof of (a).

Prove (b). It is easy to see that ${\bf J}_n^k={\bf J}_{n-1}^k\sqcup ({\bf J}_n^{k-1},n)$.
Therefore, 
$$\tilde C_n^k=\sum\limits_{\jj\in {\bf J}_n^k} y_\jj=\sum\limits_{\jj\in {{\bf J}_n^k}} y_\jj+\sum\limits_{\jj\in ({\bf J}_n^{k-1},n)} y_\jj=\tilde C_{n-1}^k+\tilde C_n^{k-1}y_{n+1-k}\ .$$

This proves (b).

To prove (c) we need the following result.

\begin{lemma} 
\label{le:Jdecomposition} 
${\bf J}_{n+1}^k=\bigsqcup\limits_{i=0}^k {\bf J}_i^i \times T^{i+1}({\bf J}_{n-i}^{k-i})$ 
for all $k,n\in \ZZ_{\ge 0}$, $0\le k\le n$, where $T=T_r:\ZZ^r\to \ZZ^r$, $r\ge 1$ is the translation given by $x\mapsto x+\underbrace{(1,\ldots,1)}_r$. 
\end{lemma}

\begin{proof} For each $\jj=(j_1,\ldots,j_k)\in {\bf J}_{n+1}^k$ denote by $i_\jj$ the largest $i\in [1,k]$ such that $j_i=i$ and set $i(\jj):=0$ if such an $i$ does not exist.
This implies that, $\{\jj\in {\bf J}_{n+1}^k:i_\jj=i\}={\bf J}_i^i\times T^{i+1}({\bf J}_{n-i-1}^{k-i})$ for all $i\in [0,k]$ (the first factor is empty for $i=0$).

The lemma is proved.
\end{proof}

Taking into account that for $\jj=(\jj',T^{i+1}(\jj''))\in {\bf J}_i^i \times T^{i+1}({\bf J}_{n-i}^{k-i})$, 
we have $y_\jj=y_{\jj'}T(y_{\jj''})$, we obtain:

$$\tilde C_{n+1}^k=\sum\limits_{\jj\in {\bf J}_{n+1}^k} y_\jj=\sum\limits_{i\in [0,k],\jj'\in {\bf J}_i^i,\jj''\in {\bf J}_{n-i}^{k-i}} 
y_{\jj'}T(y_{\jj''})=\sum_{i=0}^k \tilde C_i^iT(\tilde C_{n-i}^{k-i})\ .$$
This proves (c).

Proposition \ref{pr:recursion Cnk} is proved. \endproof

\medskip

\noindent
{\bf Proof of Theorem \ref{th:det GH}}. Applying  $\chi_q$ to $\tilde C_{n+1}^k$ given by Proposition \ref{pr:recursion Cnk}(c)  and using the fact that $\chi_q(T(y))=q^{d}\chi_q(y)$ for any  homogeneous noncommutative polynomial of degree $d$ in $y_1,y_2,\ldots$, we obtain:
$$\chi_q(\tilde C_{n+1}^k)=\sum\limits_{i=0}^k q^{k-i}\chi_q(\tilde C_i^i)\chi_q(\tilde C_{n-i}^{k-i})$$
for all $0\le k\le n$. 
In view of \cite[Equation (3.41)]{H} and  that $F_{n,k}(q,t)=H_{n,k}(q,t)=t^{n-k}q^{\frac{k(k-1)}{2}}c_{n-1}^{n-k}(q,t)$ for all $0\le k<n$, we obtain same recursion
$c_{n+1}^k=\sum\limits_{i=0}^k c_i^i(q,1) q^{k-i}c_{n-i}^{k-i}(q,1)$
for all $0\le k\le n$.
Using this and taking into account that $\chi_q(\tilde C_{n+1}^n)=\chi_q(\tilde C_{n+1}^{n+1})$, 
%Applying  $\chi_q$ to ?? and taking into account Remark ??,we obtain the following recursion:
%
%
%
%the expression for $\tilde C_n^k$ in Proposition \ref{pr:recursion Cnk}, we obtain
%$$\chi_q(\tilde C_n^k)=\sum\limits_{j_1\le \ldots \le j_k\le n:\\ j_1\ge 1,\ldots,j_k\ge k} q^{j_1+j_2+\cdots+j_k-\frac{k(k+1)}{2}}$$
%
%On the other hand, i taking into account that $c_n^k(q,1)=q^{-\frac{(n+1-k)(n-k)}{2}}H_{n+1,n+1-k}(q,1)$ and $H_{m,r}(q,1)$ is given by \cite[Equation I.23]{GH}, n view of Lemma ?? 
we conclude that $\chi_q(\tilde C_n^k)=c_n^k(q,1)$ for all $0\le k\le n$.

The theorem is proved. 
\endproof

\medskip

\noindent
{\bf Proof of Proposition \ref{pr:recursion C}}. Indeed, taking into account that $C_r=\tilde C_r^r\cdot x_0=\tilde C_r^{r-1}y_1x_0$ for all $r\ge 1$, we see that the first identity \eqref{eq:recursion C} is equivalent to
$\tilde C_{n+1}^n=\sum_{k=0}^n \tilde C_k^kT(\tilde C_{n-k}^{n-k})$
which coincides with the assertion of Proposition \ref{pr:recursion Cnk}(c) with $k=n$.

The second identity \eqref{eq:recursion C} follows from the first one and Proposition \ref{pr:bar invariance} by applying the anti-involution $\overline{\cdot}$.
%Clearly,  $P\in \PP_k$ does not contain northwest corners of the form $(r,r)$ for $0<r<k$ iff it 
%equals $((0,0),(1,0)),T(P'),(k,k-1),(k,k))$.
%for some $P'\in \PP_{k-1}$, where  $T_1$ denotes the parallel translation $\ZZ^2\to \ZZ^2$ given by $T_1(a,b)=(a+1,b)$.
 %It is clear, that $M_P=T(M_{P'})$ for all such $P\in \PP_k$.
%
%To finish the proof, for any $P\in \PP_{n+1}$  pick up its vertex $(k,k)$ with the maximal $k<n+1$. Then, clearly,$P=(P_1,T_2^k(P_2))$ for some $P_1\in \PP_k$ and $P_2\in \PP_{n+1-k}$, where $T_2$ denotes the parallel translation $\ZZ^2\to \ZZ^2$ given by $T_2(a,b)=(a+1,b+1)$.
%By definition, and the above discussion,  for such $P\in \PP_{n+1}$ we have  $M_P=M_{P_1}\cdot x_0^{-1}\cdot T(M_{P'_2})$, where $P_2'\in P_{n-k}$ is obtained from $P$ by the above procedure. 
%
%It proves the first formula \eqref{eq:recursion C}, since there is a bijective correspondence between paths $P$ and
%pairs $(P_1,P_2)$.
%
%The second identity \eqref{eq:recursion C} can be proved similarly to the first one with the choice of minimal $k>0$ such that $(k,k)\in P$. 

Proposition \ref{pr:recursion C} is proved.
\endproof

\smallskip
\noindent
{\bf Proof of Proposition \ref{pr:recursion underline C}}. We say that $x\in F$ is alternating if it is of the form 
$x_{i_1}x_{i_2}^{-1}x_{i_3}\dots x_{i_{s-1}}^{-1}x_{i_s}$ for some $i_1,\ldots,i_s\in \ZZ_{\ge 0}$ and denote by $F^{alt}$ the set of all alternating elements in $F$. We also denote by $\ZZ F^{alt}$ the $\ZZ$-linear span of $F^{alt}$ in $\ZZ F$.
We need the following fact.

\begin{lemma} 
\label{le:sigmaT}
$\sigma (T(x))=x_0\sigma(x)x_1$ for all  $x\in \ZZ F^{alt}$. 
%\begin{lemma} For $k\geq 1$ 
%$$
%\alpha (T_1(x_0^{-1}C_k))=x_1^{-1}\alpha (C_k)x_1\ . 
%$$
\end{lemma}

\begin{proof} We first prove the assertion for all $x\in F^{alt}$. Indeed,let $x=x_{i_1}x_{i_2}^{-1}x_{i_3}\dots x_{i_{s-1}}^{-1}x_{i_s}$ for some $i_1,i_2,\dots,i_s\ge 0$.
We have
$
\sigma (T(x))=\sigma (x_{i_1+1}x_{i_2+1}^{-1}x_{i_3+1}\dots x_{i_{s-1}+1}^{-1}x_{i_s+1})
$
$$
=(x_0^{i_1+1}x_1^{i_1+1})(x_0^{i_2+1}x_1^{i_2+1})^{-1}(x_0^{i_3+1}x_1^{i_3+1})\dots (x_0^{i_{s-1}+1}x_1^{i_{s-1}+1})^{-1}(x_0^{i_s+1}x_1^{i_s+1})
$$
$$
=x_0\cdot (x_0^{i_1}x_1^{i_1})(x_0^{i_2}x_1^{i_2})^{-1}(x_0^{i_3}x_1^{i_3})\dots (x_0^{i_{s-1}}x_1^{i_{s-1}})^{-1}(x_0^{i_s}x_1^{i_s})\cdot x_1
=x_0\sigma (x)x_1\ .
$$

By linearity of $\sigma$ we obtain the assertion for all $x\in \ZZ F^{alt}$.

The lemma is proved. 
\end{proof}

Since each $C_k$ belongs to $\ZZ F^{alt}$, Lemma \ref{le:sigmaT} implies that $\sigma (T(C_k))=x_0\sigma(C_k)x_1=x_0\underline C_k x_1$ for all $k\ge 0$. 
Using this and applying $\sigma$ to the first identity
\eqref{eq:recursion C}, we obtain \eqref{eq:recursion underline C}.

Proposition \ref{pr:recursion underline C} is proved.
\endproof

\medskip
\noindent
{\bf Proof of Theorem \ref{th:catalan via truncated catalans}}. In the notation of the proof of Proposition \ref{pr:recursion Cnk}, for all $0\le k\le n$ denote by 
$\overline {{\bf J}_n^k}$ the set of all $\jj=(j_1,\ldots,j_n)\in {\bf J}_n^n$ such that $j_1\ge n-k$. 

\begin{lemma}
\label{le:C_n^k} 
$\overline C_n^k\cdot x_0^{-1}=\sum\limits_{\jj\in \overline {{\bf J}_n^k}} y_\jj$ for all $0\le k\le n$.

\end{lemma}

\begin{proof}Indeed, in view of \eqref{eq:barMP}, we obtain using Lemma \ref{le:bijection paths sequences}(b):
$$\overline {C_n^k}x_0^{-1}=\sum_{P\in \PP_n^k} \overline M_P\cdot x_0^{-1}=\sum_{P\in \PP_n^k} M_{s_n(P)}\cdot x_0^{-1}=\sum_{P\in \PP_n^k} y_{\jj(s_n(P))}=\sum_{\jj\in {\overline {\bf J}_n^k}} y_{\jj}$$
because $\overline {{\bf J}_n^k}=\jj(s_n(\PP_n^k))$.

The lemma is proved.
\end{proof}

Furthermore, after multiplying by $x_0^{-1}$ on the right, the assertion of Theorem \ref{th:catalan via truncated catalans} is equivalent to:
\begin{equation}
\label{eq:tilde Cnn}
\tilde C_n^n=\sum\limits_{\substack{a,b\in \ZZ_{\ge 0}:\\a+b\leq n, a-b=d}}
\tilde C_{n-b}^a\cdot (\overline {C_{n-a}^b}x_0^{-1})
\end{equation}
for each $n\in \ZZ_{\ge 0}$ and each $d\in \ZZ$ with $|d|\le n$.

%To prove \eqref{eq:tilde Cnn} 
%if $|d|=n$, we have nothing to prove, so we fix 
%fix $d\in [1-n,n-1]$. The following is immediate.

\begin{lemma} Let $d\in [1-n,n-1]$. For each   $\jj=(j_1,\ldots,j_n)\in {\bf J}_n^n$ there exists a unique $a=a(\jj,d)\in [\max(0,d),n]$ such that $j_a\le n+d-a\le j_{a+1}$ (with the convention $j_0=0$, $j_{n+1}=\infty$).

\end{lemma}

\begin{proof} Consider graph of the linear function $y=n+d-x$ on the coordinate plain. Set $a=k$ if there exists 
$1\le k\le n$ such the point with coordinates $(k,j_k)$ is closest to the graph from the left. Otherwise
set $a=0$.
\end{proof} 

For $a\in [\max(0,d),n]$ denote by ${\bf J}_n^n(a,d)$ the set of all $\jj\in {\bf J}_n^n$ such that $a(\jj,d)=a$.

We need the following fact (in the notation of Lemmas \ref{le:Jdecomposition} and \ref{le:C_n^k}).

\begin{lemma} 
\label{le:J(a,d)}
${\bf J}_n^n(a,d)={\bf J}_{n+d-a}^a\times T^a(\overline {{\bf J}_{n-a}^{a-d}})$.

\end{lemma}

\begin{proof} 
Clearly, for any sequence $\jj\in {\bf J}_n^n(a,d)$ its subsequence $\jj'=(j_1,\dots, j_a)$
belongs to ${\bf J}_{n+d-a}^a$ and the subsequence $\jj''=(j_{a+1},\dots, j_n)$ belongs
to $T^a(\overline {{\bf J}_{n-a}^{a-d}})$.

Conversely, it is also clear that for any sequences $\jj'\in {\bf J}_{n+d-a}^a$ and $\jj''\in T^a(\overline {{\bf J}_{n-a}^{a-d}})$ their concatenation $\jj=(\jj',\jj'')$ belongs to ${\bf J}_n^n(a,d)$. 

The lemma is proved.
\end{proof}

For any two sequences of integers $\jj'=(j_1',\dots, j_k')$ and $\jj''=(j_1'',\dots, j_{\ell}'')$ define the
shifted concatenation by
$\jj'\bullet \jj'':=(\jj', T^k(\jj''))$.
We use now an obvious fact that if $\jj'=(j_1',\dots, j_k')$, $\jj''=(j_1'',\dots, j_{\ell}'')$
%$\jj$ is the concatenation of these tuples 
then
$$y_{\jj'\bullet \jj''}=y_{\jj'}y_{\jj''}\ .$$
   
Then, applying this formula to
$\jj=(\jj',T^a(\jj''))\in {\bf J}_{n+d-a}^a\times T^a(\overline {{\bf J}_{n-a}^{a-d}})$, 
we obtain:
$$\displaystyle {
\tilde C_{n+1}^n=\sum\limits_{\jj\in {\bf J}_n^n} y_\jj=
\sum\limits_{
\substack{a\in [\max(0,d),n],\\\jj'\in {\bf J}_{n+d-a}^a,\jj''\in \overline {{\bf J}_{n-a}^{a-d}}}} 
y_{\jj'}y_{\jj''}=\sum_{a\in [\max(0,d),n]} \tilde C_{n+d-a}^a\cdot (\overline {C_{n-a}^{a-d}}x_0^{-1})\ .}
$$ 
This proves \eqref{eq:tilde Cnn}.

Theorem \ref{th:catalan via truncated catalans} is proved. \endproof

%For each $a,b\in \ZZ_{\ge 0}$ with $a+b\le n$ denote by $\PP_n(a,b)$ the set of all $P\in \PP_n$ such that $(n-b,a)\in P$. 
%That is, each $P\in \PP_n(a,b)$ is of the form $P=(P',(n-b,a),P'')$, where the last point of $P'$ is $(n-b,a-1)$ and the first point of $P''$ is $(n-b+1,a)$.

\subsection{Proof of Theorems \ref{th:mult binom} and \ref{th:alternating recursion}}
\label{subsect:3.2} 
For any set $X$ and  $k\ge 0$ denote by $\left \lbrace \begin{matrix} X \\ k \end{matrix}\right\rbrace$ the
set of all subsets $J\subset X$ of cardinality $|J|=k$. 
Clearly, 
$\left \lbrace \begin{matrix} [1,m+n] \\ k \end{matrix}\right\rbrace=\bigsqcup\limits_{\substack {a,b\in \ZZ_{\ge 0}:\\a+b=k}}\left \lbrace \begin{matrix} [1,m] \\ a \end{matrix}\right\rbrace\times T^m\left(\left \lbrace \begin{matrix} [1,n] \\ b \end{matrix}\right\rbrace \right)$ for all $m,n,k\in \ZZ_{\ge 0}$ in the notation of Lemma \ref{le:Jdecomposition}, where we view each $J\in \left \lbrace \begin{matrix} [1,n] \\ k \end{matrix}\right\rbrace$ naturally as an element of $\ZZ^b$.

Taking into account that for $J=(J',T^m(J''))\in \left \lbrace \begin{matrix} [1,m] \\ a \end{matrix}\right\rbrace\times T^m\left(\left \lbrace \begin{matrix} [1,n] \\ b \end{matrix}\right\rbrace \right)$, $a+b=k$,  
we have $y_J=T^{m+a}(y_{J''})y_{J'}$ and $y'_J=T^b(y'_{J'})T^{m-a}(y'_{J''})$, we obtain for $m,n,k\in \ZZ_{\ge 0}$:
$${\stretchleftright{\llparenthesis}{\begin{matrix} m+n \\ k \end{matrix}}{\rrparenthesis}}=
\sum\limits_{J\in \left \lbrace \begin{matrix} [1,m+n] \\ k \end{matrix}\right\rbrace} y_J
=\sum\limits_{\substack{a,b\in \ZZ_{\ge 0}:a+b=k,\\J'\in \left \lbrace \begin{matrix} [1,m] \\ k \end{matrix}\right\rbrace,J''\in \left \lbrace \begin{matrix} [1,n] \\ k \end{matrix}\right\rbrace}} 
T^{m-a}(y_{J''})y_{J'}=
\sum\limits_{\substack{a,b\in \ZZ_{\ge 0}:\\a+b=k}} T^{m+a}\left({\stretchleftright{\llparenthesis}{\begin{matrix} n \\ b \end{matrix}}{\rrparenthesis}}\right) {\stretchleftright{\llparenthesis}{\begin{matrix} m \\ a \end{matrix}}{\rrparenthesis}}\ ,$$ 

$${\stretchleftright{\llparenthesis}{\begin{matrix} m+n \\ k \end{matrix}}{\rrparenthesis}}'=
\sum\limits_{J\in \left \lbrace \begin{matrix} [1,m+n] \\ k \end{matrix}\right\rbrace} y'_J
=\sum\limits_{\substack{a,b\in \ZZ_{\ge 0}:a+b=k,\\J'\in \left \lbrace \begin{matrix} [1,m] \\ k \end{matrix}\right\rbrace,J''\in \left \lbrace \begin{matrix} [1,n] \\ k \end{matrix}\right\rbrace}} 
T^b(y'_{J'})T^{m-a}(y'_{J''})=\sum\limits_{\substack{a,b\in \ZZ_{\ge 0}:\\a+b=k} }  T^b\left({\stretchleftright{\llparenthesis}{\begin{matrix} m \\ a \end{matrix}}{\rrparenthesis}}'\right) T^{m-a}\left({\stretchleftright{\llparenthesis}{\begin{matrix} n \\ b \end{matrix}}{\rrparenthesis}}'\right)\ .$$

Theorem \ref{th:mult binom} is proved. 
\endproof 

\medskip
\noindent
{\bf Proof of Theorem \ref{th:alternating recursion}}. 
%First,  after replacing $j$ by $n-j$ and $k$ by $n-k$, the assertion of the theorem becomes equivalent to 
%\begin{equation}
%\label{eq:alternating recursion modified}
%\sum\limits_{j=0}^k (-1)^j \tilde C_{n+k-j}^j\cdot {\stretchleftright{\llparenthesis}{\begin{matrix} n-j \\ k-j \end{matrix}}{\rrparenthesis}}=0 
%\end{equation}
%for any $0< k\le  n$.
%\sum\limits_{j=0}^k (-1)^j \tilde C_{2n-j}^j\cdot {\stretchleftright{\llparenthesis}{\begin{matrix} 2n-k-j \\ k-j \end{matrix}}{\rrparenthesis}}=0
%So we prove \eqref{eq:alternating recursion modified}.  
For each $0\le j\le k\le  n$ denote by ${\bf I}_{j,k;n}$ the set of all $\ii=(i_1,\ldots,i_k)\in \ZZ_{\ge 1}^k$ such that 
$i_j\le n+k+1-2j$, $i_{j+1}\le n+k-1-2j$, $i_s\le i_{s+1}+1$ for all $s\in [1,j]$,  and $i_s>i_{s+1}+1$ for all $s\in [j+1,k]$. 
(with the convention that if  $j\in \{0,k\}$, then meaningless inequalities are omitted and ${\bf I}_{-1,k;n}={\bf I}_{k+1,k;n}=\emptyset$).

The following statement is straightforward. 

\begin{lemma} 
\label{le:counting summands y} 
$\tilde C_{n+k-j}^j\cdot {\stretchleftright{\llparenthesis}{\begin{matrix} n-j \\ k-j \end{matrix}}{\rrparenthesis}}=\sum\limits_{\ii\in {\bf I}_{j,k;n}} Y_\ii$ for all $0\le j\le k$, where we abbreviate $Y_\ii:=y_{i_1}\cdots y_{i_k}$.

\end{lemma}

For $j\in [0,k+1]$ denote ${\bf I}_{j,k;n}^-={\bf I}_{j- 1,k;n}\cap {\bf I}_{j,k;n}$. 
By definition, ${\bf I}_{0,k}^-={\bf I}_{k+1,k}^-=\emptyset$ and the following is immediate.

\begin{lemma} 
\label{le:alternation cancelation}
${\bf I}_{j,k;n}^-$ is the set of all $\ii=(i_1,\ldots,i_m)\in {\bf I}_{j,k;n}$ such that $i_j\le i_{j+1}+1$ for all $j\in [0,k]$. 
In particular, ${\bf I}_{j,k;n}={\bf I}_{j,k;n}^-\sqcup {\bf I}_{j+1,k;n}^-$ for $j\in [0,k]$.

\end{lemma}

Using Lemmas \ref{le:counting summands y}  and \ref{le:alternation cancelation}, we obtain for all $0< k\le n$:
$$\sum\limits_{j=0}^k (-1)^j \tilde C_{n+k-j}^j\cdot {\stretchleftright{\llparenthesis}{\begin{matrix} n-j \\ k-j \end{matrix}}{\rrparenthesis}}=\sum\limits_{j\in [0,k],\ii\in {\bf I}_{j,k;n}} (-1)^j Y_\ii=
\sum\limits_{j\in [0,k],\ii\in {\bf I}_{j,k;n}^-} (-1)^j Y_\ii +\sum\limits_{j\in [0,k],\ii\in {\bf I}_{j+1,k;n}^-} (-1)^j Y_\ii=0 \ .$$

%, i.e.,  \eqref{eq:alternating recursion modified} is proved.

Theorem \ref{th:alternating recursion} is proved.
\endproof

\subsection{Proof of Theorems \ref{th:Gauss} and \ref{th:inverse gauss}} 
\label{subsect:3.3} 
We prove Theorem \ref{th:Gauss} first. Indeed, the assertion is equivalent to
$(H_m)_{ij}=\sum\limits_{k=0}^{\min(i,j)} (L_m)_{ik}(U_m)_{kj}$, i.e., to 
$C_{m+i+j}=\sum\limits_{k=0}^{\min(i,j)} C_{i+k+m}^{i-k}\cdot x_{2k+m}^{-1} \overline {C_{k+j+m}^{j-k}}$
for all $i,j\in \ZZ_{\ge 0}$, $m\in \{0,1\}$. This identity coincides with that from Theorem \ref{th:catalan via truncated catalans} taken with $n=m+i+j$, $a=i-k$, $b=j-k$, $d=i-j$.

Theorem \ref{th:Gauss} is proved. \endproof

\medskip
\noindent
{\bf Proof of Theorem \ref{th:inverse gauss}}. It suffices to do so only for $L^-_m$ (the argument for $U^-_m$ is identical). Indeed, the assertion is equivalent to 
$\sum\limits_{k'=i'}^{j'} (L_m)_{j'k'}(L^-_m)_{k'i'}=\delta_{i'j'}$, i.e., to 
$\sum\limits_{k'=i'}^{j'} \tilde C_{j'+k'+m}^{j'-k'}\cdot (-1)^{i'+k'}{\stretchleftright{\llparenthesis}{\begin{matrix} i'+k'+m \\ k'-i' \end{matrix}}{\rrparenthesis}}
=0$
for all $0\le i'<j'$. It is easy to show that this identity coincides with that from Theorem \ref{th:alternating recursion} taken with $n=i'+j'+m$, $j=j'-k'$ $k=j'-i'$.

%$\sum\limits_{j=0}^k (-1)^j \tilde C_{n+k-j}^j\cdot {\stretchleftright{\llparenthesis}{\begin{matrix} n-j \\ k-j \end{matrix}}{\rrparenthesis}}=0$.

Theorem \ref{th:inverse gauss} is proved. \endproof

%For $m\in \{0,1\}$ let $L^-_m$ be the lower unitriangular  $\ZZ_{\ge 0}\times  \ZZ_{\ge 0}$ matrix whose 
%$(j,i)$-th entry, $0\le i\le j$, is $(-1)^{i+j}{\stretchleftright{\llparenthesis}{\begin{matrix} i+j+m \\ j-i \end{matrix}}{\rrparenthesis}}$ and let $U^-_m$ be the upper triangular  $\ZZ_{\ge 0}\times  \ZZ_{\ge 0}$ matrix whose 
%$(i,j)$-th entry, $0\le i\le j$, is $\displaystyle{(-1)^{i+j}\overline{{\stretchleftright{\llparenthesis}{\begin{matrix} i+j+m \\ j-i \end{matrix}}{\rrparenthesis}}}x_{2j+m}^{-1}}$.
%
%%Clearly, $\displaystyle{\varepsilon (U_{k})=\varepsilon (L_{k})^T=\left((-1)^{i+j}\binom{i+j-2+k}{2i-2+k}, 1\le i \le j\right)}$.
%
%For any $\ZZ_{\ge 0}\times \ZZ_{\ge 0}$ matrix $M$ denote by $M|_n$ the principal $(n+1)\times (n+1)$-submatrix of $M$ (e.g., $H_m^n=H_m|_n$).
%%We abbreviate $L^-_m:=(L_m)^{-1}$ for $Also denote  by $L_m^n$ (resp. $U_m^n$) the principal $n\times n$ submatrix of $L_m$ (resp. of 
%
%
%
%
%
%
%
%For $m\in \{0,1\}$ let $U_m$ (resp. $L_m$) be the upper (resp. lower) triangular  $\ZZ_{\ge 0}\times  \ZZ_{\ge 0}$ matrix whose 
%$(i,j)$-th entry is $\overline {C_{i+j+m}^{j-i}}$  (resp. $\tilde C_{i+j+m}^{i-j}$). Note that $L_m$ is lower unitriangular.
%
%

\subsection{Proof of Theorems \ref{th:Garsia-Haiman det}, \ref{th:Cnk via binomials}} 
\label{subsect:3.4}
We start with a proof of Theorem \ref{th:Cnk via binomials}. The following  is well-known. 

\begin{lemma} 
\label{le:lower triangular inverse}
Any lower unitriangular $\ZZ_{\ge 0}\times \ZZ_{\ge 0}$ matrix $A=(a_{ij})$ over an associative unital ring ${\mathcal A}$ is invertible and 
$(A^{-1})_{ji}=\sum\limits_{j=i_1>i_2>\cdots >i_k=i,k\ge 1} (-1)^{k-1}a_{i_1,i_2}\cdots a_{i_{k-1},i_k}$
for all $1\le i\le j\le n$.

\end{lemma}

Applying Lemma \ref{le:lower triangular inverse} with $A=L^-_m$, i.e., $a_{ji}=\tilde C_{i+j+m}^{i-j}$ and using Theorem \ref{th:inverse gauss} in the form  $(A^{-1})_{ji}=(-1)^{i+j}{\stretchleftright{\llparenthesis}{\begin{matrix} i+j+m \\ j-i \end{matrix}}{\rrparenthesis}}$, 
we obtain the first identity. Swapping $A$ and $A^{-1}$, we obtain the second one.

Theorem \ref{th:Cnk via binomials} is proved. \endproof

\medskip
\noindent
{\bf Proof of Theorem \ref{th:Garsia-Haiman det}}. Recall from \cite{gr} that for any matrix over a commutative ring, its determinant equals the product of its principal quasiminors. Let $\underline H_m^n=\chi_q(H_m^n)=(c_{i+j+m}(q,1))$, $i,j=0,\ldots,n$, 
where $\chi_q:\ZZ F \to \ZZ[q,q^{-1}]$ is  defined in Section \ref{sect:main}.
Since all principal submatrices of $\underline H_m^n$ are $\underline H_m^k$, $k=0,1,\ldots,n$, these and Corollary \ref{cor:quasidet} guarantee that  
$\det(\underline  H_m^n)=\prod\limits_{k=0}^n \chi_q(x_{m+2k})=q^{\sum\limits_{k=0}^m\frac{(m+2k)(m+2k-1)}{2}}=q^{\frac{n(n+1)(4n-1+6m)}{6}}$.

Theorem \ref{th:Garsia-Haiman det} is proved. 
\endproof

%
%
%the determinant in left hand side vanishes for $t=q^m$, $m=0,1,\ldots,k-1$. To do so, it suffices to show that
%$${\bf r_1}=\sum_{i=0}^k \begin{bmatrix} k\\ i\end{bmatrix}_q {\bf r}_{m+1}\ ,$$
%where $\begin{bmatrix} m\\ i\end{bmatrix}_q=\frac{[m]_q!}{[i]_q![m-i]_q!}$ is the $q$-binomial coefficient and 
%${\bf r}_j$ denotes the $j$-th row of the matrix in the left hand side. 
%
%
%
%This immediately follows from the following simple combinatorial fact (which can be interpreted as counting $m\times n$-matrices over the finite field according to their ranks).
%
%\begin{lemma} For each $k\ge 1$ one has:
%$$q^{mn}=\sum_{i=0}^m \begin{bmatrix} m\\ i\end{bmatrix}_q F_m^n(q)\ ,$$
%where  $F_i^n(q):=(q^n-1)(q^n-q)\cdots (q^n-q^{i-1})$ (with the convention that $F_0^n(q)=1$).
%\end{lemma}
%
%??
%The proposition is proved. 
%\end{proof}

  \end{document}